\documentclass[11pt,letterpaper]{amsart}

\numberwithin{equation}{section}
\theoremstyle{plain}
\newtheorem{theorem}{Theorem}[section]

\newtheorem{lemma}[theorem]{Lemma}

\theoremstyle{definition}

\newtheorem{thm}{Theorem}[section]
\theoremstyle{plain}
\newtheorem{lem}[theorem]{Lemma}
\theoremstyle{plain}
\newtheorem{cor}[theorem]{Corollary}
\theoremstyle{definition}
\newtheorem{defn}[theorem]{Definition}
\theoremstyle{remark}
\newtheorem{rem}[theorem]{Remark}

\usepackage{hyperref}

\newcommand{\Rmnum}[1]{\expandafter\@slowromancap\romannumeral #1@}

\newcommand{\mr}{\mathbb{R}}
\newcommand{\ud}{\mathrm{d}}

\allowdisplaybreaks

\def\R{\mathbb R}

\def\S{\mathbb S}

\def\<{\langle}
\def\>{\rangle}

\address{Mingxiang Li, Department  of Mathematics \& Institue of Mathematical Sciences, The Chinese University of Hong Kong, Shatin, NT, Hong Kong   }
\email{mingxiangli@cuhk.edu.hk}
\address{Biao Ma, Beijing International Center for Mathematical Research, China}
\email{biaoma@bicmr.pku.edu.cn}
\begin{document}
	\title[Complete conformal metrics with prescribed Q-curvature]{Existence of complete  conformal metrics on $\mr^n$ with   prescribed Q-curvature}
	\author{Mingxiang Li, Biao Ma}
	\date{}
	\maketitle
	\begin{abstract}
		Given a smooth function $f(x)$ on $\mr^n$ which is  positive somewhere and  satisfies $f(x)=O(|x|^{-l})$ for any $l>\frac{n}{2}$, we show that there exists a complete and  conformal metric $g=e^{2u}|dx|^2$ with finite total Q-curvature such that its  Q-curvature  equals to $f(x)$.
	\end{abstract}
	
	\section{Introduction}

The prescribed curvature problems have a long and extensive history. For instance, the celebrated  Nirenberg's problem seeks a conformal metric on the standard sphere 
$(\S^2,g_0)$ such that a prescribed function 
$K(x)$ serves as the Gaussian curvature of the conformal metric 
$g=e^{2u}g_0$. Nirenberg's problem has garnered significant attention, leading to a wealth of results documented in works such as \cite{CY Acta,CY JDG, Han, JLX,KZ 74 Compact,KZ 74 open,Stru} among others. In higher dimensions, the prescribed Q-curvature problem on $\S^n$ has been explored in \cite{Brendle,Mal-Stru,WX}, with additional insights into Q-curvature provided in \cite{BCY} and \cite{CY}.

In this paper, we investigate the prescribed Q-curvature problem on complete non-compact manifolds, with the Euclidean space serving as the classical model. More specifically, we consider the existence of complete conformal metric on $\R^n$ with Q-curvature equal to a given function. In terms of differential equations, we study the following equation:
\begin{equation}\label{equ: originial eqaution}
(-\Delta)^{\frac{n}{2}}u(x) = Q(x)e^{nu(x)}, \quad \text{on} \quad x\in\mathbb{R}^n,
\end{equation}
	where $n\geq2$ is an even integer and  $Q(x)$ is a given function that serves as the Q-curvature of the  conformal metric $g=e^{2u}|dx|^2$. For the case $n=2$,	this problem has been extensively studied by Aviles \cite{Aviles}, McOwen \cite{McOwen},  Cheng-Lin \cite{Cheng-Lin MA},  \cite{Cheng-Lin Pisa} and many others. For $n\geq 4$, several existence results for equation \eqref{equ: originial eqaution} have been established by  Chang-Chen \cite{Chang-Chen}, Wei-Ye \cite{Wei-Ye}, Hyder, Mancini and Martinazzi \cite{HMM}, Hyder and Martinazzi \cite{HM}, Huang-Ye-Zhou \cite{Huang-Ye-Zhou} and  many others. See \cite{CQY} for further properties of complete metrics associated with Q-curvature on $\R^n$.  We say that the conformal metric  $g=e^{2u}|dx|^2$ with finite total Q-curvature if $Qe^{nu}\in L^1(\mr^n)$.
    
    For prescribed Q-curvature problem on $\R^n$, our result is stated as follows.
\begin{theorem}\label{thm:main theorem}
		Given a smooth function $f(x)$ on $\R^n$ which is positive somewhere and satisfies  $f(x)=O(|x|^{-l})$
		where $l>\frac{n}{2}$ and $n$ is an even integer. Then there exists a complete  and conformal metric $g=e^{2u}|dx|^2$ with finite total Q-curvature such that its  Q-curvature equals to $f$.
	\end{theorem}

We should remark that the completeness condition plays an important role here. If completeness is not assumed, the existence of solutions to \eqref{equ: originial eqaution} has been established in Theorem 2.1 of \cite{Chang-Chen} for any function $f(x)$ that is positive somewhere and satisfies the decay condition $f(x)=O(|x|^{-l})$ with $l>0$.  However, the situation becomes significantly different when we require a complete metric. In \cite{Li 24}, certain obstructions to the prescribed Q-curvature problem are identified for complete and conformal metrics on $\R^n$ with finite total Q-curvature. Specifically, when 
$f(x)$ is positive, the result can be stated as follows.

	 \begin{theorem}\label{thm obstruction}(Theorem 1.5 in \cite{Li 24})
	 		Given a  positive and smooth function  $f(x)$ on $\mr^n$ where $n\geq 2$ is an even integer and $f(x)$ satisfies 
	 	\begin{equation}\label{condition for f}
	 		\frac{x\cdot \nabla f(x)}{f(x)}\geq -\frac{n}{2}.
	 	\end{equation}
	 	Then there is no complete conformal metric $g=e^{2u}|dx|^2$ on $\mr^n$  with finite total Q-curvature  such that its  Q-curvature  eqauls to $f$.
	 \end{theorem}
	
	 Condition \eqref{condition for f} implies that $f(x)\geq C|x|^{-\frac{n}{2}}$ for $|x|\gg1$. Consequently, when the decay rate of 
$f(x)$ is slower than $|x|^{-\frac{n}{2}}$, there are obstructions to  complete metrics with finite total Q-curvature. Therefore, the result demonstrated in Theorem \ref{thm:main theorem} is sharp, in view of the decay rate.

We briefly outline our strategy. We employ the variational method to find a solution to the prescribed
Q-curvature equation \eqref{equ: originial eqaution} under an appropriate prescribed total Q-curvature, following the approach developed in \cite{McOwen}. In contrast to the two-dimensional case, our analysis requires a weighted Moser-Trudinger-Adams type inequality. Subsequently, by utilizing the distance growth identity established in \cite{Li 23 Q-curvature}, we demonstrate that, under the assumption of a specific decay rate, we can construct a complete metric.

In a recent paper \cite{Huang-Ye-Zhou}, Huang, Ye, and Zhou proved the existence of conformal metrics on $\R^n$ with prescribed non-positive Q-curvature. Their argument is based on Leray-Schauder fixed point theorem. In contrast, we require Q-curvature to be positive at some point and use the variational method.  
    
    The structure of this paper goes naturally. In section \ref{sec2}, we discuss the function spaces that are used in our variational approach. In section \ref{section 2}, we state and present a proof for a weighted Moser-Trudinger-Admas type inequality. Then in section \ref{section 3}, we finish our  proof of Theorem \ref{thm:main theorem}.

	\section{Function spaces}\label{sec2}
	In this section, we introduce the function space that fits our variational scheme.

	Let $g_{0}=|dx|^{2}$ be the standard Euclidean metric. Let $B_{1}(0)$
	be a unit ball in $\R^{n}$. Let $g_1=e^{2w}g_{0}$ be a conformal
	metric where $w$ is a smooth function on $B_{1}\backslash\{0\}$.
	We denote for $u\in C_{0}^{\infty}(B_{1}(0))$ and $i=0,1$,
	\[
	\|u\|_{W_{0}^{k,p}(B_{1}(0),g_{i})}:=\sum_{l=0}^{k}|\int_{B_{1}(0)\backslash\{0\}}|\nabla_{g_{i}}^{l}u|^{p}dV_{g_{i}}|^{\frac{1}{p}}.
	\]
	We denote $W_{0}^{k,p}(B_{1}(0),g_{i})$ to be the closure of $C_{0}^{\infty}(B_{1}(0))$
	under $\|\cdot\|_{W_{0}^{k,p}(B_{1}(0),g_{i})}$.

    The following well known Hardy-Sobolev inequality is very useful to compare function
	spaces with different norms. 
	\begin{lem}
		(Hardy-Sobolev inequality) Let $p>0$ be a positive number. Then, for $u\in W_{0}^{k,p}(B_{1})$ with $kp<n$,
		we have 
		\begin{equation}
			\left\Vert \frac{u}{|x|^{k}}\right\Vert _{L^{p}(B_{1})}\leq\|\nabla^{k}u\|_{L^{p}(B_{1})}.\label{eq:-14}
		\end{equation}
	\end{lem}

    By Hardy-Sobolev inequality, we can show the equivalence of $W^{\frac{n}{2},2}$ spaces.
	\begin{cor}
		\label{cor:Suppose-that-,}Suppose that $\alpha\in(0,1]$, $w=(\alpha-1)\log|x|+h$
		where $h$ is a smooth function in $\overline{B_{1}(0)}$. Let $g_{1}=e^{2w}|dx|^{2}$.
		Then, we have 
		\[
		W_{0}^{\frac{n}{2},2}(B_{1}(0),g_{1})\simeq W_{0}^{\frac{n}{2},2}(B_{1}(0)).
		\]
	\end{cor}
	
	\begin{proof}
		We only need to check for function $u\in C_{0}^{\infty}(B_{1})$.
		We have 
		\[
		e^{w}\nabla_{g_{1}}u=\nabla u,\ e^{2w}\nabla_{g_{1}}^{2}u=\nabla^{2}u+\<\nabla u,\nabla w\>g_{0}-\nabla u\otimes\nabla w-\nabla w\otimes\nabla u.
		\]
		For higher order derivatives,
		\[
		e^{kw}\nabla_{g_{1}}^{k}u=\nabla^{k}u+\sum_{l=0}^{k-1}\nabla^{l}u*\Phi_{k-l}(\nabla w,\cdots,\nabla^{k-l}w),
		\]
		where $\Phi_{k-l}$ is a polynomial of covariant derivatives of $w$.
		$*$ means certain linear combinations. The explicit expression of
		$\Phi_{k-l}$ will not be used, but by induction, we have estimates
		\begin{align}
		    |\Phi_{k-l}|\leq C|x|^{-(k-l)}.
		\end{align}
        
		For $k\in\{0,\cdots,\frac{n}{2}\}$, since $h$ is smooth in $\overline{B_{1}(0)}$,
		there exists a constant $C_{1}=C_{1}(n,w)$ such that $e^{(n-2k)w}\leq C_{1}|x|^{2k-n}$. Then
		\begin{align*}
			\int_{B_{1}(0)}|\nabla_{g_{1}}^{k}u|^{2}dV_{g_{1}} & =\int_{B_{1}(0)}|e^{kw}\nabla_{g_{1}}^{k}u|^{2}e^{(n-2k)w}dx\\
			% & \leq\int_{B_{1}(0)}\left(2|\nabla^{k}u|^{2}+C\sum_{l=0}^{k-1}|\nabla^{l}u|^{2}|x|^{-2(k-l)}\right)e^{(n-2k)w}dx\\
			& \leq C_{2}\int_{B_{1}(0)}\left(|\nabla^{k}u|^{2}+\sum_{l=0}^{k-1}|\nabla^{l}u|^{2}|x|^{-2(k-l)}\right)|x|^{2k-n}dx\\
			& \leq C_{3}\int_{B_{1}(0)}|\nabla^{\frac{n}{2}}u|^{2}dx.
		\end{align*}
		Here $C_2,C_3$ are constants depending on $n,k,w$ and we used Hardy-Sobolev inequality (\ref{eq:-14}) in the last line. 
		Thus, \[
		\|u\|_{W_{0}^{\frac{n}{2},2}(B_{1}(0),g_{1})}\leq C(n,p,h)\|u\|_{W_{0}^{\frac{n}{2},2}(B_{1}(0))}.
		\]
        The other direction follows from the Poincar\'e inequality.
		%Next, we see that 
		%\[
		%\|u\|_{L^{2}(dV_{g_{1}})}\geq C\|u\|_{L^{2}(dx)}.
		%\]
		%On the other hand, we have 
		%\[
		%\|u\|_{L^{2}(dV_{g_{1}})}\leq C(n,p,h)\|u\|_{L^{p}(dx)}
		%\]
		%for $p>\frac{2}{\alpha}$. The Sobolev embedding theorem shows that
		%$\|u\|_{L^{p}(dx)}\leq C(n,p)\|u\|_{W_{0}^{\frac{n}{2},2}(B_{1})}.$
		%Thus
		%		The other direction is the same. 
	\end{proof}
	The discussion above can be applied to function spaces on $\R^n$ with suitable weighted norms. Let $\epsilon>0$. Consider a positive and smooth function $\varphi:\mathbb{R}^n\to\R$ satisfying $\varphi=|x|^{-n(1+\epsilon)}$ for $|x|>1$.
	Let $\gamma=\frac{1}{n}\log \varphi $. Then 
	$e^{2\gamma}=|x|^{-2-2\epsilon},$
	for $|x|>1$. We denote 
	\begin{align}
	d\mu(x)=e^{n\gamma}dx.
	\end{align}

    Let $\sigma:\S^n\to\R^{n}$ be the standard stereographic projection
	from the north pole $x_{N}$. Let $g_{1}=\sigma^{*}\left(e^{2\gamma}|dx|^{2}\right)$.
	Then $g_{1}$ is a smooth metric on $\S^n$ away from the north pole.
	At the north pole, $g_{1}$ has a cone singularity if $\epsilon>0$. To see that, we
	take  $y=\frac{x}{|x|^{2}}.$ Then in $|y|<\frac{1}{2}$,
	we have
	\begin{equation}
		e^{2\gamma}|dx|^{2}=|y|^{2(\epsilon-1)}|dy|^{2}\label{eq:-4}
	\end{equation}
	in $B_{\frac{1}{2}}(0).$ This implies that $g_{1}$ has a cone singularity
	of index $\epsilon-1$ at the north pole. 
    
	\begin{defn}
	Notations as above.	Fix $\epsilon>0$, for $u\in C^{\infty}(\mathbb{R}^{n})$, let
		\begin{equation}
			\|u\|_{\mathcal{H}'}^{2}:=\sum_{j=0}^{\frac{n}{2}}\int_{\mathbb{R}^{n}}e^{(n-2j)\gamma}|\nabla^{j}u|^{2}dx.\label{eq:}
		\end{equation}
		Define a Hilbert space
		\[
		\mathcal{H}:=\overline{\{u\in C^{\infty}(\mathbb{R}^{n}):\|u\|_{\mathcal{H}'}<\infty\}},
		\]
		where the closure is taken under $\|\cdot\|_{\mathcal{H}'}$ norm. 
	\end{defn}
	
	We should remark that the function space $\mathcal{H}$ does not depend on the constant $\epsilon>0$. In fact, we will identify $\mathcal{H}$
	with $W^{\frac{n}{2},2}(\S^n,g_{S})$ where $g_{S}$ is the standard
	metric on $\S^n$. In order to be consistent with discussion in later sections, we denote
	\[
	\|u\|_{\mathcal{H}}^{2}:=\int_{\mathbb{R}^{n}}|(-\Delta)^{\frac{n}{4}}u|^{2}dx+\int_{\mathbb{R}^{n}}u^{2}d\mu.
	\]
	Since $\|\cdot\|_{\mathcal{H}'}$ is equivalent to $\|\cdot\|_{\mathcal{H}}$, this notation will not cause any problem.
	
	\begin{lem}
		\label{lem:There-exists-a}There exists a constant $C>0$ such that
		\[
		C^{-1}\|u\|_{\mathcal{H}}^2\leq\|u\|^2_{\mathcal{H}'}\leq C\|u\|_{\mathcal{H}}^2.
		\]
		As a result, we have 
		\[
		W^{\frac{n}{2},2}(\S^n,g_{1})\simeq W^{\frac{n}{2},2}(\S^n,g_{S})\simeq(\mathcal{H},\|\cdot\|_{\mathcal{H}})\simeq(\mathcal{H},\|\cdot\|_{\mathcal{H}'}).
		\]
	\end{lem}
	
	\begin{proof}
	We use $\sim$ to denote the equivalence of norms. By Corollary \ref{cor:Suppose-that-,} and a partition of unity argument,
		we have $\|u\|_{W^{\frac{n}{2},2}(\S^n,g_{S})}\sim\|u\|_{W^{\frac{n}{2},2}(\S^n,g_{1})}\sim\|u\|_{\mathcal{H}'}$. Thus,
    \[W^{\frac{n}{2},2}(\S^n,g_{1})\simeq W^{\frac{n}{2},2}(\S^n,g_{S})\simeq(\mathcal{H},\|\cdot\|_{\mathcal{H}'}).
		\]
 Clearly $\|u\|_{\mathcal{H}'}\geq \|u\|_{\mathcal{H}}$. It is well known (See \cite{BCY}, \cite{CY}) that 
		\begin{align*}
			\|u\|_{W^{\frac{n}{2},2}(\S^n,g_{S})}^{2} & \sim\int_{\S^n}uP_{g_{S}}udV_{g_{S}}+\|u\|_{L^{2}(dV_{g_{S}})}^{2}\\
			& \sim\int_{\mathbb{R}^{n}}|(-\Delta)^{\frac{n}{4}}u|^{2}dx+\|u\|_{L^{2}(dV_{g_{S}})}^{2}.\label{norm_eq}
		\end{align*}
So it remains to show that $\|u\|_{L^{2}(dV_{g_{S}})}^{2}\leq C\|u\|^2_{\mathcal{H}}$.		

We claim that there exists some constant $C$ such
that the following Poincar\'e type inequality holds
\begin{equation}
\|u-\bar{u}\|_{L^{2}(dV_{g_{S}})}^{2}\leq C\left(\int_{\S^n}uP_{g_{S}}udV_{g_{S}}\right),\label{eq:-17}
\end{equation}
where $\bar{u}=\text{vol}(\S^n,g_{1})^{-1}\int_{\S^n}udV_{g_{1}}$.
We argue by contradiction. Suppose there exists a sequence of $u_{i}\in\mathcal{H}$
such that 
\[
\int_{\S^n}u_{i}P_{g_{S}}u_{i}dV_{g_{S}}\to0,\ \int_{\S^n}u_{i}dV_{g_{1}}=0,\ \int_{\S^n}u_{i}^{2}dV_{g_{S}}=1.
\]
Then, as $u_{i}$ is abounded in $W^{2,\frac{n}{2}}(\S^n,g_S)$, there exists a weakly
convergent subsequence $u_{j}\rightharpoonup u_{\infty}$ and $u_{j}\to u_{\infty}$
strongly in $L^{2}$. Thus, 
\begin{align}
\int_{\S^n}u_{\infty}dV_{g_{1}} & =0,\label{eq:-15}\\
\int_{\S^n}u_{\infty}^{2}dV_{g_{S}} & =1,\label{eq:-166}
\end{align}
\begin{equation}
\int_{\S^n}|\sqrt{P_{g_{S}}}u_{\infty}|^{2}dV_{g_{S}}\leq\liminf_{j\to\infty}\int_{\S^n}u_{i}P_{g_{S}}u_{i}dV_{g_{S}}=0.\label{eq:-16}
\end{equation}
It must follow that $u_{\infty}\equiv c$. From (\ref{eq:-15}),
$u_{\infty}=0$ which contradicts to (\ref{eq:-166}). We have proved the claim.

Now, from (\ref{eq:-17}) we have 
$$\|u\|^2_{L^2(dV_{g_S})}\leq C(\bar u^2+\int_{\S^n}uP_{g_{S}}udV_{g_{S}})\leq C\|u\|^2_{\mathcal{H}}.$$ We have finished the proof.
 \end{proof}
	With Lemma \ref{lem:There-exists-a}, we will work on space $\mathcal{H}$
	with norm $\|\cdot\|_{\mathcal{H}}$. Other equivalent
	norms will be used under different circumstances.
	A function $u\in\mathcal{H}$ will be considered as a function on
	$\R^{n}$ or a function $\sigma^{*}u$ on $\S^n$.

\section{A weighted Moser-Trudinger-Admas  type inequality}\label{section 2}

 Morser-Trudinger-Adams inequalities are fundamental analytic devices in studies of Q-curvature. See, for instance, Branson-Chang-Yang \cite{BCY}, Chang-Yang \cite{CY}, Fontana \cite{Fontana}. A weighted version was first proved by Lam-Lu \cite{LamLu}, and
was used in Fang-Ma \cite{FaMa} to study constant Q-curvature metrics with conical singularities. In this section, we state and prove a weighted Morser-Trudinger-Adams type inequality that will be used in our main theorem. Such a weighted inequality is well known to experts in the field, but we have not found a precise one in the literature except in dimension 4. So, for the convenience of the readers, we take this opportunity to present a complete proof. 
 
	Let $\nabla_S$ be the standard Levi-Civita connection on round sphere $\S^n$, $\Delta_S$ be the Laplace-Beltrami operator on $\S^n$, and $P_{g_S}$ be the Panteiz operator. 
	Let $m=\left\lfloor \frac{n}{4}\right\rfloor $ and
	\[
	\square u:=\begin{cases}
		(-\Delta_{S})^{\frac{n}{4}}u, & n=4m,\\
		\nabla_{S}\left((-\Delta_{S})^{\frac{n-2}{4}}u\right), & n=4m+2.
	\end{cases}
	\]

The relevant  weighted Moser-Trudinger-Admas  type inequality states the following.
    \begin{thm}
		\label{thm:We-have}Let $u\in\mathcal{H}$. Let $\bar{u}=\frac{1}{\text{vol}(\S^n,g_{1})}\int_{\S^n}udV_{g_{1}}$. We have 
		\begin{equation}
			\log\int_{\S^n}e^{n|u-\bar{u}|}e^{n\gamma}dV_{S}\leq C+\frac{n}{2(n-1)!|\S^n|\min\{\epsilon,1\}}\int_{\S^n}|\square u|^{2}dV_{S}.\label{eq:-11}
		\end{equation}
	\end{thm}

For the main application to prescribing Q-curvature problem, we also need the following corollary.

	\begin{cor}\label{cor: weighted AMT ineq}
		There exists a constant $C=C(n,\gamma)$ such that 
		\begin{equation}
			\log\int_{\S^n}e^{n|u-\bar{u}|}dV_{g_{1}}\leq C+\frac{n}{2(n-1)!|\S^n|\min\{\epsilon,1\}}\int_{\S^n}uP_{g_{S}}udV_{S}.\label{eq:-12}
		\end{equation}
		On $\mathbb{R}^{n}$, 
		\begin{equation}
			\log\int_{\mathbb{R}^{n}}e^{n|u-\bar{u}|}e^{n\gamma}dx\leq C+\frac{n}{2(n-1)!|\S^n|\min\{\epsilon,1\}}\int_{\mathbb{R}^{n}}|(-\Delta)^{\frac{n}{4}}u|^{2}dx.\label{eq:-13}
		\end{equation}
	\end{cor}

The rest of the section is devoted to prove Theorem \ref{thm:We-have}. We first state a few technical lemmas.

	Let $G_{k}(x,y)$ be the Green function of $(-\Delta_{S})^{k}$ on
	$\S^n$. Then we have 
	\[
	u(x)-u_{S}=\begin{cases}
		\int_{\S^n}G_{m}(x,y)\square u(y)dV_{S}, & n=4m,\\
		\int_{\S^n}\nabla_{S}G_{m+1}(x,y)\cdot\square u(y)dV_{S}, & n=4m+2.
	\end{cases}
	\]
Here $u_{S}=|\S^n|^{-1}
\int_{\S^n}udV_S$.  By \cite[Lemma 2.6, 2.7]{Fontana}, we have the following expansion for Green function:
	\begin{lem}
		\label{lem:Notations-as-above.}Notations as above. Let $d(\cdot,\cdot)$
		be the distance function on $\S^n$ with $g_{S}$.
		\begin{enumerate}
			\item For $k={1,...,\frac{n}{2}-1}$,
			\[
			G_{k}(x,y)=-c_{m}\chi d(x,y)^{2k-n}(1+O(d(x,y)^{\frac{1}{2}}))+R(x,y),
			\]
			where $c_{k}=\frac{\Gamma(\frac{n}{2}-k)}{2^{2k-1}\Gamma(\frac{n}{2})\Gamma(k)|S^{n-1}|}$,
			$R$ is a smooth function on $\S^n\times \S^n$, and $\chi(x,y)$
			equals $1$ in a neighborhood of the diagonal of $\S^n\times \S^n$
			and equals $0$ when $d(x,y)>\frac{1}{4}$. 
			\item If $n=4m+2$, then
			\begin{equation}
				|\nabla_{S}G_{m+1}(x,y)|\leq c_{m}'\chi d(x,y))^{-\frac{n}{2}}(1+O(d(x,y)^{\frac{1}{2}}))+R_{1}(x,y),\label{eq:-2}
			\end{equation}
			where $c_{m}'=\frac{\Gamma(\frac{n}{2}-m)}{2^{2m}\Gamma(\frac{n}{2})\Gamma(m+1)|S^{n-1}|}$,
			$R_{1}$ is a smooth function on $\S^n\times \S^n$.
		\end{enumerate}
	\end{lem}
	
	\begin{rem}
		If $n=4m$, then $c_{m}=\frac{1}{2^{n/2-1}\Gamma(n/2)|S^{n-1}|}=\frac{1}{(2\pi)^{\frac{n}{2}}}$.
		If $n=4m+2$, we have $c_{m}'=\frac{1}{(2\pi)^{\frac{n}{2}}}.$
	\end{rem}

    From Green functions and their expansions, we can solve poly-harmonic equations and obtain standard $L^p$ estimates which can be summarized in the following Lemma.
	\begin{lem}
		\label{lem:Let-.-If}Let $q>1$. If $n=4m$, for $f\in L^{q}(\S^n,dV_{S})$,
		there exists a function $g$ in $W^{2m,q}(\S^n,g_{S})\cap L^{2}(\S^n,dV_{S})$
		such that 
		\[
		(-\Delta_{S})^{m}g=f,
		\]
		and $\|g\|_{L^{2}}\leq c\|f\|_{L^{q}}$ for some positive constant $c=c(n,q)$.
		If $n=4m+2$, for $f\in L^{q}(\S^n,dV_{S})$, there exists a function
		$g$ in $W^{2(m+1),q}(\S^n,dV_{S})$ such that 
		\[
		(-\Delta_{S})^{m+1}g=f,
		\]
		and $\|\nabla_{S}g\|_{L^{2}}\leq c'\|f\|_{L^{q}}$ for some positive constant
		$c'(n,q)$.
	\end{lem}
	
	% \begin{proof}
	% 	Assume $n=4m$. By standard $L^{p}$ theory for elliptic equation,
	% 	there exists a solution $g\in W^{2m,q}(\S^n,g_{S})$ which satisfies
	% 	\[
	% 	g(x)=\int_{\S^n}G_{m}(x,y)f(y)dV_{S}(y).
	% 	\]
	% 	Let $p=2$, $\frac{1}{r}+\frac{1}{p}+\frac{1}{q}=2$. Notice that
	% 	if $q>1$, then $0<r<2$. From the expansion of $G_{m},$ $\|G_{m}(x,\cdot)\|_{L^{r}(dV_{S})}$
	% 	is uniformly bounded by some constant $c(n,q)$. Then, by Young's
	% 	inequality, we have
	% 	\begin{align}
	% 		\int_{\S^n}g^{2}dV_{S} & =\int_{\S^n}\int_{\S^n}g(x)G_{m}(x,y)f(y)dV_{S}(x)dV_{S}(y)\label{eq:-1}\\
	% 		& \leq c(n,q)\|g\|_{L^{2}}\|f\|_{L^{q}}.\nonumber 
	% 	\end{align}
	% 	Thus, $\|g\|_{L^{2}}\leq c(n,q)\|f\|_{L^{q}}$. 
		
	% 	Assume $n=4m+2$. Then there exists a solution $g\in W^{2m+2,q}(\S^n,g_{S})$
	% 	which satisfies
	% 	\[
	% 	g(x)=\int_{\S^n}G_{m+1}(x,y)f(y)dV_{S}(y).
	% 	\]
	% 	Now, $\|\nabla_{S}G_{m}(x,\cdot)\|_{L^{r}(dV_{S})}$ is also uniformly
	% 	bounded due to (\ref{eq:-2}). By (\ref{eq:-1}), $\|\nabla_{S}g\|_{L^{2}(dV_{S})}$
	% 	is also bounded by $c'\|f\|_{L^{q}}$ for some constant $c'(n,q)$.
	% \end{proof}
	
	We also need the following estimate.
	\begin{lem}
		\label{lem:Fix-.-Let}Fix $x\in \S^n$. Let $B_{\delta}(x)$ be a
		geodesic ball of radius $\delta<\frac{\pi}{2}$ centered at $x$.
		Let $r=d_{\S^n}(x,y)$. Let $u\in\mathcal{H}$. We have 
		\begin{equation}
			|u(x)-\bar{u}|\leq\frac{1}{(2\pi)^{\frac{n}{2}}}\int_{B_{\delta}(x)}r^{-\frac{n}{2}}|\square u(y)|dV_{S}+C\|\square u\|_{L^{2}(dV_{S})}.\label{eq:-3}
		\end{equation}
		Here $C$ is a constant depending only on $n,\gamma,\epsilon$.
	\end{lem}
	
	\begin{proof}
		Let 
		\[
		\rho=\frac{e^{n\gamma}}{\text{vol}(\S^n,g_{1})}-\frac{1}{|\S^n|}.
		\]
		Then $\int_{\S^n}\rho dV_{S}=0$ and $\rho\in L^{q}(\S^n,dV_{S})$
		for $1<q<\frac{1}{1-\epsilon}$ if $\epsilon<1$ and $\rho\in L^{\infty}(\S^n)$
		if $\epsilon\geq1$. Let $\psi$ be the solution to 
		\[
		\begin{cases}
			(-\Delta_{S})^{m}\psi=\rho, & \mathrm{if }\ n=4m,\\
			(-\Delta_{S})^{m+1}\psi=\rho, & \mathrm{if }\ n=4m+2.
		\end{cases}
		\]
		Then $\|\psi\|_{L^{2}(dV_{S})}$ is bounded by $C(n,\gamma,\epsilon)$
		when $n=4m$. $\|\nabla_{S}\psi\|_{L^{2}(dV_{S})}$ is bounded by
		$C(n,\gamma,\epsilon)$ when $n=4m+2$. We have 
		\[
		u(x)-\overline{u}=\begin{cases}
			\int_{\S^n}\left(G_{m}(x,y)-\psi(y)\right)\square u(y)dV_{S}, & n=4m,\\
			\int_{\S^n}\left(\nabla_{S}G_{m+1}(x,y)-\nabla_{S}\psi(y)\right)\cdot\square u(y)dV_{S}, & n=4m+2.
		\end{cases}
		\]
		If $n=4m$, by Lemma \ref{lem:Let-.-If}, 
		\[
		\left|\int_{\S^n}\psi(y)\square u(y)dV_{S}\right|\leq\|\psi\|_{L^{2}(dV_{S})}\|\square u\|_{L^{2}(dV_{S})}.
		\]
		If $n=4m+2$, by Lemma \ref{lem:Let-.-If}, 
		\[
		\left|\int_{\S^n}\nabla_{S}\psi(y)\cdot\square u(y)dV_{S}\right|\leq\|\nabla_{S}\psi\|_{L^{2}(dV_{S})}\|\square u\|_{L^{2}(dV_{S})}.
		\]
		Then (\ref{eq:-3}) follows from the expansion of $G_{m}$ and $\nabla G_{m+1}$
		in Lemma \ref{lem:Notations-as-above.}. 
	\end{proof}
	Now, we state a weighted Moser-Trudinger-Adams inequality proved by 
	Lam-Lu \cite{LamLu}. 
	\begin{thm}[Lam-Lu]
		\label{thm:Let-.-Let}Let $\beta\in(-1,0]$. Let $\Omega$ be a open
		set in $\mathbb{R}^{n}$ with finite lebesgue measure. Then, there
		exists a constant $c_{0}=c_{0}(p,\Omega)$ such that for all $f\in L^{2}(\mathbb{R}^{n})$
		supported in $\Omega$,
		\[
		\int_{\Omega}\exp\left((1+\beta)b_{n}\frac{\left|I_{\frac{n}{2}}*f(x)\right|^{2}}{\|f\|_{L^{2}}^{2}}\right)|x|^{n\beta}dx\leq c_{0},
		\]
		where $I_{\frac{n}{2}}*f(x)=\frac{1}{(2\pi)^{\frac{n}{2}}}\int_{\Omega}|x-y|^{-\frac{n}{2}}f(y)dy,$
		and $b_{n}=\frac{n(2\pi)^{n}}{|S^{n-1}|}$.
	\end{thm}
	
Combining arguments  in \cite{BCY,FaMa,Fontana}, we now 
	prove the main theorem of this section.

	\begin{proof}[Proof of Theorem \ref{thm:We-have}]
		Let $f=\<\square u(y),\square u(y)\>_{g_S}^{\frac{1}{2}}$. We first
		assume that 
		\[
		0<\|f\|_{L^{2}(\S^n,dV_{S})}\leq1.
		\]
		Pick a small geodesic ball $B_{\delta}(x_{N})$ centered at the north
		pole $x_{N}$ on $\S^n$. Using a normal coordinate at $x_{N}$,
		we may write $g_{ij}(y)=\delta_{ij}+O(|y|^{2})$. Define
		\[
		u_{\delta}(x):=\frac{1}{(2\pi)^{\frac{n}{2}}}\int_{B_{\delta}(0)}|x-y|^{-\frac{n}{2}}f(y)dy.
		\]
		Then, from Lemma \ref{lem:Fix-.-Let}, we have 
		\begin{equation}
			|u(x)-\bar{u}|\leq u_{\delta}(x)+C\|f\|_{L^{2}(\S^n,dV_{S})}.\label{eq:-5}
		\end{equation}
		Choosing a smaller $\delta,$ we may further assume that $\|f\|_{L^{2}(B_{\delta})}\leq1$.
		Let $\alpha=\min\{\epsilon,1\}$. From Theorem \ref{thm:Let-.-Let},
		we have
		\begin{equation}
			\int_{B_{\delta}(0)}\exp\left(\alpha b_{n}u_{\delta}^{2}\right)|y|^{n(\alpha-1)}dy\leq c_{0}.\label{eq:-4-1}
		\end{equation}
		From the metric expansion (\ref{eq:-4}), (\ref{eq:-5}), and (\ref{eq:-4-1}),
		we have 
		\begin{equation}
			\int_{B_{\delta}(x_{N})}\exp\left(\alpha b_{n}|u-\bar{u}|^{2}\right)dV_{g_{1}}\leq c_{0}'(n,\gamma).\label{eq:-6}
		\end{equation}
		For any non-constant function $v\in W^{\frac{n}{2},2}(\S^n,g_{S})$,
		we can apply (\ref{eq:-6}) to $\frac{v}{\|\square v\|_{L^{2}}}$.
		By mean value inequality, we have 
		\begin{equation}
			\int_{B_{\delta}(x_{N})}\exp\left(\alpha b_{n}\frac{|v-\bar{v}|^{2}}{\|\square v\|_{L^{2}}^{2}}+\frac{n^{2}}{4b_{n}\alpha}\cdot\|\square v\|_{L^{2}}^{2}\right)dV_{g_{1}}\geq\int_{B_{\delta}(x_{N})}\exp\left(n|v-\bar{v}|\right)dV_{g_{1}}.\label{eq:-7}
		\end{equation}
		We apply (\ref{eq:-7}) to $u$ to have 
		\begin{equation}
			\int_{B_{\delta}(x_{N})}\exp\left(n|u-\bar{u}|\right)dV_{g_{1}}\leq c_{0}'\exp\left(\frac{n^{2}}{4b_{n}\alpha}\cdot\|f\|_{L^{2}}^{2}\right)\text{vol}(B_{\delta}(x_{N}),g_{1}).\label{eq:-9}
		\end{equation}
		The metric $g_{1}$ on $\S^n\backslash B_{\delta/2}(x_{N})$ is
		uniformly equivalent to $g_{S}$. Thus, we may apply the standard
		Moser-Trudinger-Adams inequality proved by Fontana \cite{Fontana} on manifolds which shows
		\begin{align}
		 \label{eq:-8}   \int_{\S^n\backslash B_{\delta/2}}\exp(n|u-\bar{u}|)dV_{g_{1}}&\leq C(n,\gamma)\exp\left(\frac{n^{2}}{4b_{n}}\|f\|_{L^{2}}^{2}\right).
		\end{align}
		By (\ref{eq:-9}), (\ref{eq:-8}), and a partition of unity argument,
		we have  
		\begin{equation}
			\int_{\S^n}\exp(n|u-\bar{u}|)dV_{g_{1}}\leq C(n,\gamma)\exp\left(\frac{n^{2}}{4b_{n}\alpha}\cdot\|\square u\|_{L^{2}}^{2}\right).\label{eq:-10}
		\end{equation}
		Note $\frac{n^{2}}{4b_{n}\alpha}=\frac{n}{2(n-1)!|\S^n|\alpha}$.
		Thus, we have proved the theorem.
	\end{proof}

	To prove Corollary \ref{cor: weighted AMT ineq}, we use the same argument in Chang-Yang \cite[Lemma 1.6]{CY} to relate (\ref{eq:-11})
		and (\ref{eq:-12}). 
	\begin{proof}[Proof of Corollary \ref{cor: weighted AMT ineq}]
		 $\sqrt{P}$ has the same leading term as $\square$
		and the Green function of $\sqrt{P}$ has an expansion as in Lemma
		\ref{lem:Notations-as-above.}. Namely, the Green function $G_{\sqrt{P}}(x,y)$
		of $\sqrt{P}$ satisfies 
		\[
		|G_{\sqrt{P}}(x,y)|\leq\frac{1}{(2\pi)^{\frac{n}{2}}}\chi d(x,y)^{-\frac{n}{2}}\left(1+O(d(x,y)^{\frac{1}{2}})\right)+R(x,y),
		\]
		where $R(x,y)$ is a bounded function on $\S^n\times \S^n$. Then,
		we have (\ref{eq:-12}), following the same argument in the proof
		of Theorem \ref{thm:We-have}.
		
		The proof of (\ref{eq:-13}) follows from the conformal covariance
		of the Paneitz operator. 
	\end{proof}

	\section{Proof of the main theorem}\label{section 3}
	Denote $d\mu=h(x)dx$ on $\mr^n$ where $h(x)$ is a positive and smooth function and satisfies 
	\begin{equation}\label{weighted h}
		C^{-1}|x|^{-n-n\epsilon}\leq h(x) \leq C|x|^{-n-n\epsilon}, |x|\geq 1 
	\end{equation} for some positive constant $C$ and $\epsilon>0.$
	Let $\mathcal{H}$ denote the space with norm
	$$\|u\|^2_{\mathcal{H}}=\int_{\mr^n}((-\Delta)^{\frac{n}{4}}u)^2 d x+ \int_{\mr^n}u^2 d \mu$$
	with  the convention $(-\Delta)^{\frac{1}{2}}:=\nabla $. Meanwhile, we set 
	$$\tilde{\mathcal{H}}=\{u\in\mathcal{H}|\int_{\mr^n}ud \mu=0\}.$$ For our main application, we will consider functions in $\mathcal{H}$ and $\tilde{\mathcal{H}}$. The following Lemma gives $L^p$ estimates for functions in $\mathcal{H}$. The proof follows easily from Corollary \ref{cor: weighted AMT ineq}.

	\begin{lemma}\label{lem: weight sobolev}
		For any $v\in \mathcal{H}$ and $p>1$, there exists some positive constant $C$ such that 
		\begin{equation}
		    \|v\|_{L^p(d\mu)}\leq C\|v\|_{\mathcal{H}}. \label{tech_lem}
		\end{equation}
	\end{lemma}
% 	\begin{proof}
%     Let $N=\|(-\Delta)^{n/4}v\|_{L^2(\R^n)}$ and $u=\frac{v-\bar v}{N}$ where $\bar v:=\frac{\int_{\mr^n}v\ud \mu}{\int_{\mr^n}\ud \mu}$. From Corollary \ref{cor: weighted AMT ineq}, 
%     $$ \int_{\R^n} e^{n|u| }d\mu\leq C.$$ Thus, $\|u\|_{L^p(d\mu)}\leq C'$ which deduces that
%     \begin{equation}\label{vL^p leq vL^1}
%         \|v\|_{L^p(\ud\mu)}\leq C\|(-\Delta)^{n/4}v\|_{L^2(\R^n)}+C\|v\|_{L^1(\ud\mu)}.
%     \end{equation}
% With the help of H\"older's inequality, one has
% $$\|v\|_{L^1(\ud\mu)}\leq (\int_{\mr^n}v^2\ud\mu)^{\frac{1}{2}}(\int_{\mr^n}\ud\mu)^{\frac{1}{2}}\leq C\|v\|_{L^2(\ud\mu)}.$$
% Combining with the estimate \eqref{vL^p leq vL^1}, we finish the proof.
%       \end{proof}

    Now we can prove the existence of solutions to \eqref{equ: originial eqaution} via variational method.
	\begin{theorem}\label{thm: alpha range}
		If  a smooth  function $Q(x)$ is positive somewhere  and $$Q(x)=O(|x|^{-l}),$$ where $l>0$, 
		  for each $\alpha$ satisfies
          \begin{equation}\label{alpha range}
              \max\{0, 2-\frac{2l}{n}\}<\alpha<2,
          \end{equation}
	then there exists  a solution $u$  to \eqref{equ: originial eqaution} satisfying
		\begin{align}
		    \int_{B_R(0)}|u(x)|\ud x =o(R^{n+1}), \label{eq:integral_est}
		\end{align} and
		\begin{align}
		    \int_{\R^n}Qe^{nu}\ud x=\frac{(n-1)!|\mathbb{S}^n|}{2}\alpha. \label{eq:total_Q}
		\end{align}
	\end{theorem}
	
	\begin{proof}
		Choose a smooth function $u_0$ satisfying $u_0=-\alpha\log|x|$ for $|x|>1/2$ and define $$\psi:=(-\Delta)^{\frac{n}{2}}u_0.$$
		Notice that $\psi\equiv 0$ for $|x|>1/2$. 
		In particular, with the help  of divergence theorem,  it is not hard to verify that 
		$$\int_{\mr^n}\psi dx=\int_{B_1(0)}(-\Delta)^{\frac{n}{2}}u_0\ud x=-\int_{\partial B_1(0)}\frac{\partial (-\Delta)^{\frac{n}{2}-1}u}{\partial n}\ud \sigma.$$
		A direct computation yields that
		\begin{equation}\label{int psi}
			\int_{\mr^n}\psi dx=\frac{(n-1)!|\mathbb{S}^n|}{2}\alpha.
		\end{equation}
		Then the equation \eqref{equ: originial eqaution} is equivalent to
		\begin{equation}\label{equ: equivalent eqaution}
			(-\Delta)^{\frac{n}{2}}w+\psi=Ke^{nw},
		\end{equation}
		where
		$K(x):=Q(x)e^{nu_0}=O(|x|^{-l-n\alpha}).$
        
       Set
        \begin{equation}\label{epsilon def}
            \epsilon=\alpha+\frac{l}{n}-1>0
        \end{equation}
        due to our assumption \eqref{alpha range}.
        Set the  function $h(x)$ as follows 
 $$h(x)=(1+|x|^2)^{-\frac{n}{2}-\frac{n\epsilon}{2}}.$$
   Set  the weighted measure $\ud \mu=h(x)\ud x$. 
		Meanwhile, one has
        \begin{equation}\label{equ: Ke^nu leq e^nu mu}
            \int_{\mr^n}K(x)e^{nu}\ud x\leq C\int_{\mr^n}e^{nu}\ud\mu. 
        \end{equation}
	Consider the following set
		$$\mathcal{H}_K=\{u\in \mathcal{H}|\int_{\mr^n}Ke^{nu}\ud x>0\}.$$
		Since $K(x)$ is positive somewhere, it is not difficult to check that $\mathcal{H}_K$ is not empty.
		For any $u\in \mathcal{H}_K$, we consider the  following functional 
		$$F(u)=\frac{1}{2}\int_{\mr^n}((-\Delta)^{\frac{n}{4}}u)^2 d x+\int_{\mr^n}\psi u dx-\frac{(n-1)!|\mathbb{S}^n|}{2n}\alpha\log\int_{\mr^n}Ke^{nu} d x.$$
		For any constant $c$, by \eqref{int psi}, we have
		$$F(u+c)=F(u).$$
		Thus we may assume that 
		\begin{equation}\label{bar u=0}
			\int_{\mr^n}ud\mu=0. 
		\end{equation}
		Now, apply Young's inequality to obtain that 
		$$|\int_{\mr^n}\psi u dx|=|\int_{\mr^n}\psi u h^{-1}d\mu|\leq C(\eta)\|\psi h^{-1}\|_{L^2(d\mu)}+\eta\int_{supp(\psi)}u^2d \mu$$
		where $\eta>0$ is small to chosen later and $supp(\psi)$ denotes the compact support of the function $\psi$. Then making use of \eqref{equ: Ke^nu leq e^nu mu},  Corollary \ref{cor: weighted AMT ineq}  and Lemma \ref{lem: weight sobolev}, we  obtain that 
		$$F(u)\geq \left(\frac{1}{2}-\frac{\alpha}{4\min\{1,\epsilon\}}-C_1\eta\right)\int_{\R^n}((-\Delta)^{\frac{n}{4}}u)^2 d x-C(\eta, \alpha,l,n)$$
		Based on the assumption \eqref{alpha range} and \eqref{epsilon def}, 	we can choose small $\eta>0$ such that
	$$\frac{1}{2}-\frac{\alpha}{4\min\{1,\epsilon\}}-C_1\eta>0.$$
	With the help of  a standard variational method, there exists a function  $v\in \mathcal{H}_K$  attains the minimum of $F(u)$ which weakly satisfies 
	\begin{equation}\label{minimum equation}
		(-\Delta)^{\frac{n}{2}}v+\psi =\frac{(n-1)!|\mathbb{S}^n|\alpha}{2}\frac{Ke^{nv}}{\int_{\R^n}Ke^{nv}\ud x}.
	\end{equation}
	We can choose  a suitable constant $C_2$ 
	\begin{equation}\label{equ v_1}
		(-\Delta)^{\frac{n}{2}}v_1+\psi=Ke^{nv_1}.
	\end{equation}
where $v_1=v+C_2$
and 
$$\int_{\mr^n}Ke^{nv_1}\ud x=\frac{(n-1)!|\mathbb{S}^n|}{2}\alpha.$$
	Standard elliptic estimate shows that $v_1$ is actually smooth.
Since $v_1\in \mathcal{H}$, Lemma \ref{lem: weight sobolev} yields that for any $q>1$, there holds
\begin{equation}\label{v L^q}
    \|v\|_{L^q(d\mu)}\leq C\|v\|_{\mathcal{H}}<+\infty.
\end{equation}
Making use of H\"older's inequality, one has
$$
    \int_{B_R(0)}|v_1|\ud x\leq (\int_{B_R(0)}|v|^qh\ud x)^{\frac{1}{q}}(\int_{B_R(0)}h^{-\frac{1}{q-1}}\ud x)^{\frac{q-1}{q}}\leq CR^{n+\frac{n\epsilon}{q}}.
$$
 By choosing suitable large $q>1$, one has
 \begin{equation}\label{v_1 grwoth}
     \int_{B_R(0)}|v_1|\ud x=o(R^{n+1}).
 \end{equation}

	Finally,  choose 
	$$u(x)=v_1(x)+u_0(x)$$ which is our desired solution to \eqref{equ: originial eqaution}.
Based on the estimate \eqref{v_1 grwoth}, 	it is easy to see
	$$\int_{B_R(0)}|u(x)|\ud x=o(R^{n+1}),$$ and
	$$\int_{\R^n}Qe^{nu}\ud x=\int_{\R^n}Ke^{nv_1}\ud x=\frac{(n-1)!|\mathbb{S}^n|}{2}\alpha.$$
We have proved the theorem.
	\end{proof}

\begin{rem}
    With some further effort, one can show that $v_1(x)$ is uniformly bounded. The argument is similar to Theorem 2.2 in \cite{McOwen}. The estimates required in weighted Sobolev spaces are provided in \cite[Theorem 5]{McOwen2} and \cite[Theorem 5.4]{cantor}.
\end{rem}

To finish our proof of Theorem \ref{thm:main theorem}, we need two results in the first author's work \cite{Li 23 Q-curvature}. Recall the definition of normal solutions in \cite{CQY}: a solution $u(x)$ to \eqref{equ: originial eqaution} is \emph{normal} if $Qe^{nu}\in L^1(\mr^n)$ and  $u(x)$ satisfies the integral equation
$$u(x)=\frac{2}{(n-1)!|\mathbb{S}^n|}\int_{\mr^n}\log\frac{|y|}{|x-y|}Q(y)e^{nu(y)}\ud y+C.$$
The first result that we need characterizes normal solutions from an integral estimate.
\begin{theorem}\label{thm: normal solution}
	(Theorem 2.2 in \cite{Li 23 Q-curvature}) If the solution $u(x)$ to \eqref{equ: originial eqaution} satisfies  $Qe^{nu}\in L^1(\mr^n)$ and 
		$$\int_{B_R(0)}|u|\ud x=o(R^{n+2}), $$ then $u(x)$ is normal.
\end{theorem}
For normal solutions, the second result in \cite{Li 23 Q-curvature} gives a characterization for complete metrics. 
\begin{theorem}\label{thm: complete}
	(Theorem 1.4 in \cite{Li 23 Q-curvature}) If $u(x)$ is a normal solution satisfying \eqref{equ: originial eqaution}, and $g=e^{2u}|dx|^2$, then for each fixed point $p$, we have
	$$\lim_{|x|\to\infty}\frac{\log d_g(x,p)}{\log|x-p|}=\max\left\{1-\frac{2}{(n-1)!|\mathbb{S}^n|}\int_{\mr^n}Qe^{nu}\ud x,0\right\},$$
	where $d_g(\cdot, \cdot)$ is the geodesic distance on $\mr^n$ respect to the metric $g$. In particular, $g$ is complete if the right hand side is strictly positive.
\end{theorem}

Now we are ready to give the proof of Theorem \ref{thm:main theorem}.
\begin{proof}[Proof of Theorem \ref{thm:main theorem}]
Since $l>\frac{n}{2}$, one has
$\max\{0, 2-\frac{2l}{n}\}<1.$
By Theorem \ref{thm: alpha range}, for some 
$\alpha$ satisfying 
$$\max\{0, 2-\frac{2l}{n}\}<\alpha<1,$$
there exists a smooth solution $u(x)$ to the following equation 
$$(-\Delta)^{\frac{n}{2}}u=fe^{nu}.$$
In addition,  $u$ satisfies \eqref{eq:integral_est} and \eqref{eq:total_Q}. Then Theorem \ref{thm: normal solution} implies that $u$ is normal and  Theorem \ref{thm: normal solution}  show that $g=e^{2u}|dx|^2$ is complete. Thus we have finished our proof.
\end{proof}

%\textbf{Notations}
% \begin{enumerate}
%     \item $\mu$ measure $d\mu=h(x)dx$
%     \item $\mathcal{H}$ Normed space equivalent to $W^{2,n/2}$. $\tilde{\mathcal{H}}$  subspace of functions of  zero average under $d\mu $. 
%     \item $g_S$  round metric on $\S^n$
%     \item $\nabla_S$  Levi-Civita connection on $\S^n$, $\Delta_S $ Laplace-Beltrami on $\S^n$
%     \item $\bar{u}$  average under $d\mu$
%     \item $e^\gamma=h^{1/n}$ , $g_1=e^{2\gamma}g_S$ is a conical metric on $\S^n$.
%     \item      $\mathbb{S}^n$ sphere.

% \end{enumerate}


\begin{thebibliography}{99}
		\bibitem{Aviles}
		P. Aviles,  Conformal complete metrics with prescribed nonnegative Gaussian curvature in $R^2$,  Invent. Math. 83 (1986), no. 3, 519–544. 
		\bibitem{BCY}
	T. 	Branson, S.-Y.  Chang, and P. Yang,  Estimates and extremals for zeta function determinants on four-manifolds,  Comm. Math. Phys. 149 (1992), no. 2, 241–262.
			\bibitem{Brendle}
		S. Brendle, Convergence of the Q-curvature flow on $S^4$, Adv. Math. 205 (2006), no. 1, 1–32.
        \bibitem{cantor} M. Cantor. Elliptic operators and the decomposition of tensor fields, Bull. Amer. Math. Soc. (N.S.), 5 (1981), no. 3,  235-262.
		\bibitem{Chang-Chen}
		S.-Y. A. Chang and W.  Chen,  A note on a class of higher order conformally covariant equations, Discrete Contin. Dynam. Systems 7 (2001), no. 2, 275–281.
				\bibitem{CQY}
		S.-Y. A. Chang, J. Qing and P. C. Yang,  On the Chern-Gauss-Bonnet integral for conformal metrics on $\mathbb{R}^4$, 
		Duke Math. J. 103, No. 3, 523-544 (2000). 
		\bibitem{CY Acta}
		S.-Y. A.  Chang, P. C. Yang,  Prescribing Gaussian curvature on $S^2$,  Acta Math. 159 (1987), no. 3-4, 215–259. 
		\bibitem{CY JDG}
		S.-Y. A.  Chang, P. C. Yang, Conformal deformation of metrics on $S^2$,  J. Differential Geom. 27 (1988), no. 2, 259–296.
		\bibitem{CY}A. S.-Y. Chang and P. C. Yang, Extremal metrics of zeta function determinants on  4 -manifolds, Ann. of Math. (2) 142 (1995), no. 1, 171–212.
        \bibitem{Cheng-Lin MA}
        K.-S.  Cheng and C.-S.   Lin, On the asymptotic behavior of solutions of the conformal Gaussian curvature equations in $\mathbb{R}^2$, Math. Ann. 308, No. 1, 119-139 (1997).
		\bibitem{Cheng-Lin Pisa}
		K.-S.  Cheng and C.-S.   Lin,  Compactness of conformal metrics with positive Gaussian curvature in $R^2$, Ann. Scuola Norm. Sup. Pisa Cl. Sci. (4) 26 (1998), no. 1, 31–45.
		\bibitem{FaMa}
		H. Fang and B. Ma, Constant  Q-curvature metrics on conic 4-manifolds,	Adv. Calc. Var. 15 (2022), no. 2, 235–264.
		\bibitem{Fontana}
		L. Fontana,   Sharp borderline Sobolev inequalities on compact Riemannian manifolds,  Comment. Math. Helv. 68 (1993), no. 3, 415–454.
		\bibitem{Han}
		Z.-C. Han, 		Prescribing Gaussian curvature on $S^2$,	Duke Math. J.61(1990), no.3, 679–703.
        \bibitem{Huang-Ye-Zhou}
      X. Huang, D. Ye, and F. Zhou,  Normal conformal metrics with prescribed Q-Curvature in $\mr^{2n}$, arXiv:2502.16881.
      \bibitem{HMM}
      A. Hyder, G. Mancini and L. Martinazzi, Local and nonlocal singular Liouville equations in Euclidean spaces, Int. Math. Res. Not. 15 (2021), 11393-11425.
      \bibitem{HM}
      A. Hyder and L. Martinazzi, Normal conformal metrics on $\mathbb{R}^4$ with Q-curvature having power-like growth, J. Diﬀerential Equations 301 (2021), 37-72.
		\bibitem{JLX}
		T. Jin, Y. Li and J.  Xiong, 	The Nirenberg problem and its generalizations: a unified approach, 	Math. Ann.369(2017), no.1-2, 109–151.
			\bibitem{KZ 74 Compact}
		J. 	Kazdan and F.  Warner, Curvature functions for compact 2-manifolds,  Ann. of Math. (2) 99 (1974), 14–47. 
		\bibitem{KZ 74 open}
		J. Kazdan and F.  Warner, Curvature functions for open 2-manifolds,  Ann. of Math. (2) 99 (1974), 203–219. 
		\bibitem{LamLu}
		N. Lam and G.  Lu,  Weighted Moser-Onofri-Beckner and logarithmic Sobolev inequalities,  J. Geom. Anal. 28 (2018), no. 2, 1687–1715. 

			\bibitem{Li 23 Q-curvature}
		M. Li,  The total Q-curvature, volume entropy and polynomial growth polyharmonic functions, Adv. Math. 450 (2024), Paper No. 109768, 43 pp.
		\bibitem{Li 24}
		M. Li, Obstructions to prescribed Q-curvature of complete conformal metric on $\mr^n$, arXiv:2401.03457. 
		\bibitem{Mal-Stru}
	A. 	Malchiodi and M. Struwe, 	Q-curvature flow on $S^4$,	J. Differential Geom.73(2006), no.1, 1–44.
			\bibitem{McOwen}
		R. McOwen, Conformal metrics in $R^2$ with prescribed Gaussian curvature and positive total curvature,  Indiana Univ. Math. J. 34 (1985), no. 1, 97–104.
        \bibitem{McOwen2}McOwen, R.C. , The behavior of the laplacian on weighted sobolev spaces. Comm. Pure Appl. Math., 32 (1979), 783-795.
		\bibitem{Stru}
	M. Struwe, 	A flow approach to Nirenberg's problem,	Duke Math. J. 128 (2005), no. 1, 19–64.
	\bibitem{WX}
	J. Wei and X.  Xu, 	On conformal deformations of metrics on $S^n$, J. Funct. Anal.157(1998), no.1, 292–325.
    \bibitem{Wei-Ye}
    J. Wei and D. Ye, Nonradial solutions for a conformally invariant fourth order equation in $\R^4$, Calc. Var.
P.D.E. 32(3) (2008), 373-386.
	\end{thebibliography}
\end{document}